\documentclass[11pt]{article}

\usepackage{amsmath,amsthm,thmtools,amsfonts,amssymb,csquotes,epsfig, titlesec, epsfig, float, caption,wrapfig, mathrsfs}
\usepackage[hidelinks]{hyperref}
\usepackage{tikz-cd}
\usepackage{tikz}
\usepackage{tikz-3dplot}
\usepackage{xcolor}
\usepackage{verbatim}
\usepackage{accents}
\usepackage{marvosym}
\usepackage{multirow}
\usepackage{multicol}
\usepackage{listings}
\usepackage[citestyle=alphabetic,bibstyle=alphabetic,backend=bibtex, maxbibnames=10,minbibnames=5]{biblatex}
\usepackage[colorinlistoftodos]{todonotes}
\usepackage{enumitem}
\usepackage[american]{babel}

\usetikzlibrary{decorations.markings}
\usetikzlibrary{calc, decorations.pathreplacing,shapes.misc}
\usetikzlibrary{decorations.pathmorphing}
\usepackage[left=1in,top=1in,right=1in]{geometry}
\usepackage{subcaption}
\usetikzlibrary{shapes.geometric}

\usepackage[capitalize]{cleveref}

\newtheorem{theorem}{Theorem}[section]

\newtheorem{lemma}[theorem]{Lemma}
\newtheorem{proposition}[theorem]{Proposition}

\newtheorem{conjecture}[theorem]{Conjecture}
\newtheorem{question}[theorem]{Question}

\theoremstyle{remark} 

\newtheorem{example}[theorem]{Example}

\newtheorem{remark}[theorem]{Remark}
 \crefname{rem}{Remark}{Remarks}
 \Crefname{rem}{Remark}{Remarks}

\theoremstyle{definition} 
 
\newtheorem{definition}[theorem]{Definition} 

\titleformat*{\section}{\normalsize \bfseries \filcenter}
\titleformat*{\subsection}{\normalsize \bfseries }
\newtheorem{mainthm}{Theorem}
\Crefname{mainthm}{Theorem}{Theorems}

\Crefname{maincor}{Corollary}{Corollaries}

\Crefname{mainquestion}{Question}{Questions}

\Crefname{mainconj}{Conjecture}{Conjectures}

\crefformat{equation}{(#2#1#3)}

\makeatletter
\def\namedlabel#1#2{\begingroup
   \def\@currentlabel{#2}
   \label{#1}\endgroup
}
\makeatother

\DeclareFontFamily{U}{mathx}{}
\DeclareFontShape{U}{mathx}{m}{n}{<-> mathx10}{}
\DeclareSymbolFont{mathx}{U}{mathx}{m}{n}
\DeclareMathAccent{\widehat}{0}{mathx}{"70}
\DeclareMathAccent{\widecheck}{0}{mathx}{"71}

\renewbibmacro{in:}{}

\title{\normalsize \textbf{Rigidity and Realizability for Tropical Curves in Dimension 3}}
\author{\normalsize Jeff Hicks}
\date{}

\newcommand{\eps}{\varepsilon}
\newcommand{\RR}{\mathbb R}
\newcommand{\ZZ}{\mathbb Z}
\newcommand{\NN}{\mathbb N}
\newcommand{\CC}{\mathbb C}

\newcommand{\II}{\mathbb I}

\newcommand{\CP}{\mathbb{CP}}

\newcommand{\into}{\hookrightarrow}
\newcommand{\tensor}{\otimes}

\newcommand{\ot}{\leftarrow}

\newcommand{\CF}{CF^\bullet}

\DeclareMathOperator{\id}{id}
\DeclareMathOperator{\coker}{coker}

\DeclareMathOperator{\Ext}{Ext}
\DeclareMathOperator{\Abd}{Abd}

\newcommand{\syza}{\pi_A}
\newcommand{\syzb}{\pi_B}

\DeclareMathOperator{\Coh}{Coh}
\DeclareMathOperator{\st}{\mid}

\DeclareMathOperator{\tropb}{TropB}

\DeclareMathOperator{\ev}{ev}
\DeclareMathOperator{\Def}{Def}

\DeclareMathOperator{\Area}{Area}

\DeclareMathOperator{\Grd}{Grd}

\def\ii{\mathbf{i}}
\def\jj{\mathbf{j}}
\def\aa{\mathbf{a}}
\def\x{\mathbf{x}}
\def\y{\mathbf{y}}
\def\b{\mathbf{b}}
\def\p{p}

\def\uM{\underline{\mathcal M}}
\def\upi{\underline{\pi}}
\def\uev{\underline{\ev}} 

\newcommand{\Addresses}{{\bigskip
  \footnotesize
  \noindent J.~Hicks, \textsc{School of Mathematics and Maxwell Institute for Mathematical Sciences, University of Edinburgh}\par\nopagebreak
  \noindent \textit{E-mail address}: \texttt{jeff.hicks@ed.ac.uk}
}}

\begin{document}
\maketitle
\begin{abstract}
  We present an unobstructedness criterion for Lagrangian threefolds $L\subset X^A$ using the $H_1(L)$-class associated with the boundary of a pseudoholomorphic disk.
  As an application, let $X^A\to Q$ be a Lagrangian torus fibration whose base $Q$ is a tropical abelian threefold. Given $V\subset Q$ a rigid tropical curve with a pair-of-pants decomposition, we prove that the Lagrangian lift $L_V\subset X^A$ is unobstructed. 
  Provided that an appropriate homological mirror symmetry statement holds, this implies the existence of a realization $Y_V$ in the mirror abelian threefold $X^B\to Q$.
\end{abstract}

\section{Introduction}
Deformation theory is a powerful proof technique in geometry. The general strategy is first to exhibit a solution to some problem in a special, potentially simpler, setting. Then one proves the solution can be ``deformed'' to the general case. Here are two instances where this strategy is employed:

\begin{description}
  \item[Tropical geometry:] Let $X^B\to Q$ be a tropicalization map. 
  Given a tropical subvariety $V\subset Q$, methods from deformation theory establish whether $V$ can be lifted to an algebraic subvariety  $Y_V\subset X^B$.
  \item[Symplectic geometry:] Given a Lagrangian submanifold $L$ of a symplectic manifold $X^A$, the Floer cochains $\CF(L)$ are deformation of the cochains $C^\bullet(L)$. They provide a powerful set of invariants of the pair $L\subset X^A.$
\end{description} 

Mirror symmetry is a conjectured dictionary from string theory \cite{candelas1991pair} between symplectic geometry on $X^A$ and algebraic geometry on a ``mirror space'' $X^B$. Geometrically, the spaces $X^A, X^B$ arise as dual torus fibrations over a common base $X^A\to Q\ot X^B$ \cite{strominger1996mirror}. Broadly speaking, this paper uses mirror symmetry to solve deformation theory problems when the Lagrangian $L_V\subset X^A$ is related to a tropical subvariety $V\subset Q$.

\subsection{Background: deformation problems in tropical and symplectic geometry}

Let $Q$ be an integral metric affine manifold and denote by $N_Q \subset TQ$ its integrable full-rank lattice. The space $X^B_\CC:= TQ/N_Q$ is naturally a complex manifold equipped with a holomorphic volume form that makes the torus fibers of the map $\syzb: X^B_\CC \to Q$ totally real submanifolds.

Under suitable conditions, the $\syzb$-image of a complex subvariety $Y_\CC \subset X^B_\CC$ can be approximated by a piecewise linear (tropical) subset of $Q$. To precisely establish this correspondence between subvarieties of $X^B_\CC$ and tropical geometry on $Q$, we replace the complex space with a rigid analytic space $X^B$ and a tropicalization map $\tropb: X^B \to Q$. 
The image of a subvariety $Y \subset X^B$ under the tropicalization map is a tropical subvariety in $Q$ \cite{groves1984geometry}. A natural question is the extent to which the tropical geometry of $Q$ captures the geometry of $X^B$. We say that a tropical subvariety $V \subset Q$ is \emph{$B$-realizable} if there exists an analytic subvariety $Y_V \subset X^B$ such that $\tropb(Y_V) = V$.

When $V \subset \RR^n$ is a $k$-dimensional tropical pair-of-pants, there is no obstruction to building a lift of $V$. However, even if $V \subset Q$ is locally modeled on tropical pairs of pants, there may be obstructions to globally gluing together the lifts of local pieces. The first example where this occurs is the superabundant tropical elliptic curve, which is a 3-valent smooth tropical curve in $\RR^3$ \cite{mikhalkin2005enumerative}. Understanding how combinatorial properties of $V$ translate into $B$-realizability has been extensively studied. It is already a rich problem for 3-valent tropical curves in $Q = \RR^n$, where it has been proven that simply-connectedness \cite{nishinou2006toric},  well-spacedness conditions \cite{speyer2014parameterizing,katz2012lifting,ranganathan2017skeletons,lam2024combinatoricshurwitzdegenerationstropical}, and non-superabundance \cite{cheung2016faithful} are realizability criteria for smooth trivalent tropical curves. However, there are many curves whose realizability status is not yet known (see, for example \cite{koyama2024constructions}) and nearly all research on realizability has been restricted to the setting of tropical curves or hypersurfaces of $\RR^n$.

\subsubsection*{Deformation Theory on the \texorpdfstring{$B$}{B}-side}

Given $V \subset Q$ a tropical curve, \cite{nishinou2006toric} lays out the following strategy for producing a $B$-realization. First, construct a toric degeneration $X^B_t$ dual to the tropical curve $V$. From this data, one can build a nodal curve $Y_{V, 0} \subset X^B_0$ whose restriction to each toric component is a genus zero curve, and whose intersection complex is the tropical curve $V$. Realizing the tropical curve $V$ is equivalent to finding a simultaneous deformation of the pair $Y_{V, t} \subset X^B_t$.
Logarithmic geometry provides a homology theory with obstruction classes suitable for this deformation problem, and \cite{nishinou2006toric,cheung2016faithful} obtain their realizability criteria by proving these classes vanish.

An alternate sheaf-theoretic approach to this deformation problem can also be stated: we promote $Y_{V_0}$ to the sheaf pushforward of the structure sheaf, $\mathcal O_{V,0} \in \Coh(X^B_0)$. The obstructions to deformation of the pair $\mathcal O_{V, 0}\in \Coh^{dg}(X^B_0)$ live in $\Ext^2_{X^B_0}(i_*\mathcal O_{V,0}, i_*\mathcal O_{V,0})$ \cite[Theorem 7.3]{hartshorne2010deformation}, so we might hope to leverage the vanishing of these groups to obtain $B$-realizability results. Let us now look at two examples that we will revisit throughout the introduction:

\begin{description}
  \item[B.1] Let $V \subset Q = \RR^n$ be a point---this should be one of the ``easiest'' examples for $B$-realizability. However, $\Ext^\bullet_{X^B_0}(\mathcal O_{V, 0}, \mathcal O_{V, 0})$ is the exterior algebra and has many potential obstructions in second cohomology.
  \item[B.2] Let $V\subset Q = \RR^n$ be a smooth tropical hypersurface. Then $\mathcal O_{V,0}$ admits a length 1 resolution by line bundles on $X^B_0$, from which we immediately obtain $\Ext^2_{X^B_0}(i_*\mathcal O_{V,0}, i_* \mathcal O_{V,0}) = 0$. This indicates that there are no obstructions, reflecting that all tropical hypersurfaces are $B$-realizable.
\end{description}

\subsubsection*{Deformation Theory on the \texorpdfstring{$A$}{A}-side}
We now describe a similar deformation theory problem that occurs in symplectic geometry.
Let $X$ be a compact symplectic manifold. Given a graded spin Lagrangian submanifold $L\subset X$, the Fukaya $A_\infty$ algebra $(\CF(L), m^k)$ is a deformation of the differential graded algebra of cochains on $L$. 
We will be using the setup from \cite{fukaya2010cyclic} where the underlying cochains are de Rham cochains. 
To first approximation, operations $m^k: \CF(L)^{\otimes k}\to \CF(L)[2-k]$ are defined by counting pseudoholomorphic disks with boundary on $L$. This deformation is frequently unwieldy: for instance, the filtered $A_\infty$ relations do not require that the differential relation $m^1\circ m^1=0$ holds. As a consequence, $\CF(L)$ does not usually have well-defined homology groups. The failure of the differential relation is captured by the curvature term $m^0: \Lambda_0 \to CF^2(L)$ which satisfies the weaker condition $m^1\circ m^1 = m^2(m^0 \otimes \id) \pm m^2(\id\otimes m^0)$. The lowest-order term of the right-hand side is closed and this component gives a class in $H^2(L)$ obstructing the existence of homology groups.

Several hypotheses may be used to obtain a well-defined homology theory associated with the filtered $A_\infty$ algebra $\CF(L)$. The strongest asks for $m^0=0$; then we say that $\CF(L)$ is tautologically unobstructed. A weaker condition, introduced by \cite{fukaya2010lagrangian}, is to ask that there exists a bounding cochain $b\in CF^1(L)$ so that the $b$-deformed $A_\infty$ structure has vanishing curvature term $m^0_b=0$ (see the discussion preceding \cref{eq:maurerCartan}). In this case, we say that $L$ is unobstructed (by the bounding cochain $b$) with well defined Floer complex $\CF(L, b)$. The input data for many constructions in symplectic geometry are well-defined Floer cohomology groups. The most relevant example where Floer cohomology groups play a role for this paper is the homological mirror symmetry conjecture of \citeauthor{kontsevich1994homological}. 

We now discuss two topological criteria on the pair $(X, L)$ which imply unobstructedness of $\CF(L)$. The products can be refined to $m^k= \sum_{\beta\in H_2(X, L)} m^k_\beta$ where the $\beta$-index records the relative homology class of the pseudoholomorphic disks contributing to the $A_\infty$-products.
\begin{description}
  \item[A.1] \emph{Bounds no pseudoholomorphic disks:} Suppose that $\omega(\beta)=0$ for all $\beta\in H_2(X, L)$. This occurs, for example, when our Lagrangian submanifold is exact or if $H_2(X, L)$ vanishes. As every $J$-pseudoholomorphic disk must have a positive symplectic area for $\omega$-compatible $J$, we conclude that  $m^k_\beta =0$ for all $\beta\neq 0$. Then $(\CF(L), m^k)=(C^\bullet(L), d, \wedge)$, and our Lagrangian is \emph{tautologically} unobstructed. 
  \item[A.2]  \emph{$H^2(L)$ vanishes:} \cite{fukaya2010lagrangian} introduced a strategy for producing a bounding cochain using the energy filtration on $\CF(L)$. Suppose that $\beta$ is a minimal energy disk class (in the sense $\omega(\beta)\leq \omega(\beta')$ for all $\beta'\neq 0$). Then the $\beta$-component of the curvature term is closed. If we additionally know that $H^2(L)=0$, the $\beta$-component of the curvature term $m^0_{\beta}$ is exact. Let $db_0=m^0_{\beta}$; $b_0$ is a bounding cochain to the lowest order. From this, we can recursively build a bounding cochain for $L$ order-by-order. This method, for instance, shows that all Lagrangian spheres are unobstructed.
\end{description}

\subsubsection*{Mirror Symmetry}
The deformation stories on the $A$ and $B$ sides are connected via mirror symmetry. We summarize the construction from the introduction of \cite{hicks2022realizability}.
Let $M_Q = N_Q^* \subset T^*Q$ be the dual lattice to $N_Q$. Then $X^A:= T^*Q/M_Q$ is naturally a symplectic manifold with a Lagrangian torus fibration $\syza: X^A \to Q$. 
Given a tropical subvariety $V \subset Q$, a Lagrangian submanifold $L_V \subset X^A$ is said to be a geometric realization of $V$ if $\syza(L_V)$ approximates $V$ in a precise sense \cite[Definition 3.0.2]{hicks2022realizability}.
In contrast to the algebraic-geometric side, for \emph{any} smooth trivalent tropical curve $V \subset Q$, \cite{mikhalkin2018examples,matessi2018lagrangian,hicks2020tropical} construct examples of Lagrangian submanifolds $L_V \subset X^A$ that geometrically realize $V$. There are also constructions for geometric lifts of $V$ within $Q$ beyond trivalent curves and hypersurfaces (see \cite{hicks2020tropical,mak2020tropically,haney2024cusped}). \cite{hicks2022realizability} conjectures that \emph{all} $V$ admit geometric lifts. It is clear that an additional condition is necessary to make the Lagrangian realizability problem reflect the $B$-realizability problem, and a natural candidate is Floer theoretic unobstructedness of $L_V$. We, therefore, say that a tropical subvariety is $A$-realizable if it has an unobstructed Lagrangian lift $L_V$.

The spaces $X^A$ and $X^B$ are called the $A$-side and $B$-side of an SYZ mirror pair \cite{strominger1996mirror}.
In \cite{abouzaid2017family}, \citeauthor{abouzaid2017family} showed that $X^A$ and $X^B$ satisfy a form of the homological mirror symmetry conjecture of \cite{kontsevich1994homological}. Using this framework, \cite{hicks2022realizability} demonstrated that, under a hypothesis on the construction of family Floer theory, an $A$-realization $L_V \subset X^A$ can be used to build a $B$-realization $Y_V \subset X^B$.

Let us summarize how the two examples we covered before are related under this duality:
\begin{description}
  \item[A/B.1]\emph{Bounds no pseudoholomorphic disks:} If $q\in \RR^n$ is a point, then $L_q=\syza^{-1}(q)$ is a Lagrangian torus and $H_2(X^A, L_q)=0$. So for purely topological reasons, the Lagrangian submanifold cannot bound any holomorphic disks and is tautologically unobstructed. Similarly, a tropical subvariety $V\subset \RR^n$ with a single vertex centered at the origin has an exact Lagrangian lift $L_V\subset X^B=T^*T^n$ (\cite[Claim 4.2.3]{hicks2022realizability}), and is tautologically unobstructed.
  \item[A/B.2]\emph{$H^2(L)$ vanishes:} When $H^2(L_V)=0$ we have unobstructedness using the method from \cite{fukaya2010lagrangian}. This occurs, for instance, when $Q=\RR^n$ and $\dim(V)=n-1$, thus recovering the original realizability results of \cite{mikhalkin2005enumerative}. A generalization of this argument observes that the obstruction term $m^0$ lives in the \emph{compactly supported} cochains of $L_V$, so the vanishing of the map $H^2_{cpt}(L_V)\to H^2(L_V)$ is sufficient to prove that $L_V$ is unobstructed---this holds, for instance, when $V\subset Q$ is a genus zero tropical curve \cite{hicks2022realizability}.
\end{description}
In contrast to the deformation theory on sheaves, the first case (lifting a point) is substantially easier to prove than the second case.

\subsection{Summary of Results}
This paper looks at a rigidity --- a combinatorially defined criterion  from tropical geometry --- and applies it to the unobstructedness problem for Lagrangian submanifolds. For expositional purposes, we first discuss unobstructedness before moving on to the interactions between tropical and symplectic geometry.
In \cref{sec:unobstructedness}, we provide an unobstructedness criterion for $L$ using the cyclic $A_\infty$ structure constructed by \citeauthor{fukaya2010cyclic} in \cite{fukaya2010cyclic} and the $H_1(L)$ grading. Then in \cref{sec:realizability} we show that when $L_V$ is a Lagrangian lift of a rigid tropical curve $V$ in a tropical abelian threefold, $L_V$ satisfies this unobstructedness condition.

\subsubsection*{From Unobstructedness...}
The topological structure we will use in this paper to obtain unobstructedness is the class $\partial \beta\in H_1(L)$. Informally, in complex dimension three the moduli space of pseudoholomorphic disks $\mathcal M_{\beta, 0}(X, L)$  is zero-dimensional, and so (to first approximation)
\[m^0_\beta = \textbf{PD}([\partial \beta])\cdot \#\mathcal M_{\beta, 0}(X, L)/S^1.\] 
where the $S^1$ action comes from domain reparameterization of the marked output point, and $\textbf{PD}$ is Poincar\'e duality.
Suppose we have the additional hypothesis that, whenever $u: (D^2, \partial D^2)\to (X, L)$ is a pseudoholomorphic disk, $[\partial u]=0 \in H_1(L)$. Then we can conclude (for any $\beta$) that $[m^0_\beta]=0$; this is a condition we can bootstrap to obtain unobstructedness.  We extend this to the setting where the statement holds in the limit of a sequence of complex structures. 
\begin{mainthm}
  \label{thm:unobstructedness}
  Let $\frac{1}{2}\dim_\RR(X)=3$, and let $L\subset X$ be a graded oriented compact Lagrangian submanifold. Let $\{J_E\}$ be a sequence of almost complex structures so that the limit loop grading group  $\Grd_\infty(L, \{J_E\})\subset H_1(X)$ (see \cref{def:filteredLoopGrading}) vanishes.
   Then $L$ is unobstructed by a bounding cochain. 
\end{mainthm}
The main component of the proof is \cref{lem:zeropairing}, which is a minor generalization of the divisor axiom from \cite{fukaya2010cyclic} for the Fukaya algebra. 

\subsubsection*{...to Realizability.}
One of the broadest realizability criteria is non-superabundance, originally stated for tropical curves by \cite{mikhalkin2005enumerative}.
To state this condition we need to provide some notation for tropical subvarieties. 

A tropical subvariety is a regular polyhedral complex $V$ of $Q$ (satisfying a balancing condition that we will not use). We will generally use the symbols $\sigma, \tau$ to index the strata of $V$ unless they are zero or one-dimensional in which case we will reserve the symbols $v, e$. We will use $|V|, |\sigma|, \ldots$ to denote the geometric realization of these sets, for example, $\underline i_v: |v|\to Q$ is a map whose image is a point.  As an exception, we will abuse notation and write $V\subset Q$ instead of the map $\underline i_V: |V|\into Q$.
Recall that $N_Q$ is the sheaf of integral vector fields on $Q$, which gives rise to a sheaf $\underline i_V^*N_Q$ on $|V|$. We will abuse notation and simply denote this sheaf on $|V|$ by $N_Q$ when it is clear.   
For each $\sigma$, we have an integrable sublattice $N_\sigma \subset T|\sigma|$ satisfying the property that under the natural identification $T|\sigma|\subset TQ$ we have $N_\sigma\subset  N_Q$.  For $\sigma$ any stratum of $V$, let $V_\sigma$ be the star of the stratum (which is a polyhedral subcomplex). The sets $|V_v|$ form a cover of $|V|$.

We have a restriction map $\bigoplus_{v\in V} H^0(|V_v|,N_Q)\to \bigoplus_{e\in V} H^0(|V_e|,N_Q)$; let us denote the cokernel of this map by $A(V,N_Q)$. At each edge we have an inclusion $H^0(|V_e|, N_e)\hookrightarrow H^0(|V_e|,N_Q)$, allowing us to construct an exact sequence 
\begin{equation}
  \label{eq:superabundance}
  0\to \Def(V, Q)\to \bigoplus_{e\in V}H^0(|e|, N_e) \xrightarrow{\phi}     A(V,N_Q) \to \Abd(V)\to 0.
\end{equation}
where the kernel and cokernel of $\phi$ are called the deformation and superabundance space of $V$.
\begin{definition}
  \label{def:rigid}
  We say that $V$ is rigid if $\Def(V, Q)=0$ and non-superabundant if $\Abd(V)=0$. 
\end{definition}
Observe that under the simplifying assumptions that $\dim(V)=1$ and $Q=\RR^n$ we have $A(V,N_Q)= H^1(|V|,N_Q)=H^1(|V|, \ZZ^n)$ which agrees with the formulation of non-superabundance due to \cite{katz2012lifting}.
A specialization of a theorem of \cite{cheung2016faithful} states whenever $V\subset \RR^n$ is a non-superabundant curve admitting a pair of pants decomposition, $V$ is realizable.

The following inequality provides a necessary condition for non-superabundance for a tropical curve with a pair-of-pants decomposition:
\begin{align*}
0=&\dim(\Abd(V))\geq \dim(H^1(|V|,N_Q)) - \dim ( \bigoplus_{e\in V}H^0(|e|, N_e)) \\
 =& \frac{(n-3)(\text{\# compact edges})- 2n(\text{\# semi-infinite edges})}{3}+1 
 \stepcounter{equation}\tag{\theequation} \label{eq:superabundanceInequality}
\end{align*}
Observe from this inequality that a curve $V$ will dramatically fail to be superabundant whenever $n\geq 3$ and $Q$ is compact. 
In fact, $\Abd(V)$ will always contain a copy of $H_1(Q)$, so different techniques are required to prove the realizability of curves inside abelian varieties.
In the compact setting, \cite{nishinou2020realization} provides a realizability condition for tropical curves in tropical abelian surfaces; see also \cite[Corollary D]{hicks2022realizability} using mirror symmetry.

In this paper, we turn our focus to the group $\Def(V, Q)$ instead, as when $\Def(V, Q)=0$ the group $\Abd(V, N_Q)$ is as small as possible. For example, when $Q$ is a tropical abelian three-fold the relation \cref{eq:superabundanceInequality} reduces to $\dim(\Abd(V))\geq 3+ \dim(\Def(V, Q))$, with no dependence on the number of edges.
\begin{mainthm}
  \label{thm:realizable}
  Let $Q$ be a tropical abelian threefold. Let $V\subset Q$ be a tropical curve admitting a pair of pants decomposition. Suppose that $V$ is rigid.
  Then $L_V$ is unobstructed by a bounding cochain.
\end{mainthm}
The main ingredients in the proof of \cref{thm:realizable} in addition to \cref{thm:unobstructedness} are two results for tropical Lagrangian submanifolds \emph{of any dimension}:
\cref{lem:trivialRestrictiontoVertex} which restricts the boundary homology classes of pseudoholomorphic curves using tropicalization, and \cref{lem:rigidGrading} computes the first cohomology of Lagrangian lifts of rigid tropical subvarieties. 

We conclude the paper with \cref{sec:examples},which includes some examples of rigid tropical curves.

\subsection{Limitations of our Approach}
There are several limitations to our approach.
For example, the condition that $\dim(Q)=3$ in \cref{thm:realizable} comes from our inability to straightforwardly generalize \cref{thm:unobstructedness} beyond dimension three. However, we expect that the moduli spaces of pseudoholomorphic disks with boundary on a tropical Lagrangian are sufficiently well-behaved to conjecture the following:
\begin{conjecture}
  \label{conj:higherDimension}
  Let $Q$ be a tropical abelian variety, and $V\subset Q$ be a rigid tropical curve admitting a pair-of-pants decomposition. Then $L_V$ is unobstructed by a bounding cochain.
\end{conjecture}
Additionally, we may not immediately apply the results of \cite{hicks2022realizability} to obtain a $B$-realizability condition, as the Floer cohomology considered in that paper was based on a Morse cochain/ stabilizing divisor model from \cite{charest2019floer}, while the unobstructedness criterion we prove here uses Floer theory with de Rham cochains/ Kuranishi structures \cite{fukaya2010cyclic}. 
However, it is natural to conjecture:
\begin{conjecture}
  \label{conj:bRealizability}
  Let $V\subset Q$ be as in \cref{conj:higherDimension}. Then $V$ is $B$-realizable.
\end{conjecture}
\begin{remark}
  During the preparation of this article, we learned that \cite{haney2024infinity} constructed a weakened version of the cyclic structure for the Fukaya algebra with underlying cochains given by Morse theory. While the paper restricts to the setting where evaluation maps are submersions (as in \cite{solomon2016point}), extending \citeauthor{haney2024infinity}'s work to the domain-dependent perturbation construction of \cite{charest2019floer} would allow us to deduce a $B$-model realizability criterion.
\end{remark}
Instead of applying homological mirror symmetry to deduce $B$-realizability of $V$ from \cref{thm:realizable}, we could instead write an algebraic-geometric proof of \cref{thm:realizable}.
Both \cref{thm:unobstructedness} and \cref{lem:rigidGrading} can be stated in terms of the $B$-side deformation problem. However, the key piece---\cref{lem:trivialRestrictiontoVertex} which constrains the homological grading of the deformation---remains elusive.  
\begin{question}
  What is the algebraic-geometric analog of using a limiting set of almost-complex structures on the $A$-side in terms of deformation theory?
\end{question}

\subsection{Acknowledgements}
I would like to thank \'Alvaro Mu\~niz Brea, Yusuf Bar{\i}\c{s} Kartal, and Nick Sheridan for the many helpful discussions around this topic;  Dhruv Ranganathan for answering my many questions about tropical geometry; and Andrew Hanlon for providing feedback on an early version of this paper.
I learned of the tropicalization method used in \cref{lem:trivialRestrictiontoVertex} from Brett Parker at the 2019 MATRIX workshop ``Tropical Geometry and Mirror Symmetry.'' 
Finally, I must thank Francesca Carocci, who pointed out to me that all tropical curves in Abelian varieties are superabundant; this significantly changed the goal of this article to focus on rigidity instead.
This work was supported by EPSRC Grant EP/V049097/1.

\section{\texorpdfstring{$A$}{A}-side: Unobstructedness from Loop Grading}
\label{sec:unobstructedness}
In this section we provide an unobstructedness criteria. In \cref{subsec:cyclicIntroduction} we review the construction of the cyclic $A_\infty$ structure on Floer cochains following \cite{fukaya2010cyclic}. Following that section, we provide a minor generalization of the divisor axiom in \cref{subsec:divisorAxiom} and use that to prove an unobstructedness criterion in \cref{subsec:unobstructednessCriterion}.
\subsection{Cyclic Structures in Floer Cohomology}
\label{subsec:cyclicIntroduction}
We recall some of the structure of the Fukaya $A_\infty$ algebra from \cite{fukaya2010cyclic}.
Let $L_V\subset X^A$ be a graded spin Lagrangian submanifold and $J$ a choice of $\omega$-compatible almost complex structure.  The Fukaya $A_\infty$ algebra is a gapped filtered $A_\infty$ algebra where:
\begin{itemize}
  \item the cochains $\CF(L, J)$ are de Rham cochains with values in $\Lambda_0$, the universal Novikov ring $\Lambda_{\geq 0}:=  \left\{\sum_{i=0}^\infty  a_i T^{E_i}\st a_i\in \CC, E_i \in \RR_{\geq 0}, \lim_{i\to\infty} E_i = \infty\right\}.$
  \item the filtered $A_\infty$ product is a collection of maps $m^k_\beta: \CF(L, J)^{\otimes k}\to T^{\omega(\beta)}\CF(L, J)$ indexed by $\beta\in H^2(X, L)$ described below.
\end{itemize}
The geometric intuition behind the definition of these products is that for each $\beta\in H_2(X, L)$ and $k\in \NN$ we construct the space 
\[\uM_{k, \beta}(X, L, J):=\{(D^2, \underline \theta, u) \st u:(D^2, \partial D^2)\to (X, L) m u^*J = j_u, [u]=\beta\}/\sim\]
 of pseudoholomorphic disks with $k$-marked boundary points modulo domain reparameterization. Here, $\underline \theta= \theta_0< \theta_1< \cdots < \theta_{k}$ is a set of $k$-marked ordered boundary points on the disk. For each $0 \leq i \leq k$, there is an evaluation map $\uev_i: \uM_{k, \beta}(X, L, J)\to L$.
 
 For the sake of exposition, we now assume that all moduli spaces are transversely cut-out manifolds with corners, all evaluation maps are submersions, and that all perturbations chosen preserve symmetries and forgetful maps. These assumptions rarely hold and in \cref{subsubsec:fuk10} we outline the approach from \cite{fukaya2010cyclic} to overcome this problem. If these assumptions held, we would for any $\beta\neq 0$ define the $A_\infty$ product structure by 
\begin{equation}
  \label{eq:firstApproximation}
  m^k_\beta(\x_1\otimes \cdots \otimes \x_k) \text{``:=''}(\uev_0)_*\left(\bigwedge_{0 < i \leq k} \uev_i^*\x_i\right)
\end{equation}
where $(\uev_0)_*$ is integration along the fiber of the evaluation map at the output marked point. For $\beta=0$, we use the differential graded algebra structure coming from differential forms.
\begin{align*}
  m^1_0=d, & & m^2_0=\wedge, & & m^k_0=0 \text{ otherwise.}
\end{align*}

\subsubsection{Properties of the Fukaya Algebra}
We now describe the expected properties of the maps $m^k_\beta: \CF(L, J)^{\otimes k}\to \CF(L, J)$ from our placeholder definition \cref{eq:firstApproximation}. By definition, $\CF(L, J)$ is a deformation of the de Rham DGA. 
An immediate consequence of \cref{eq:firstApproximation} is that when the moduli space $\uM_{k, \beta}(X, L, J)$ is empty the associated product vanishes:
\begin{equation}
  \label{eq:emptyVanishing}
  m^k_\beta=0 \text{ whenever }  \uM_{k, \beta}(X, L, J) =\emptyset , \beta\neq 0.
\end{equation}
The boundary compactification of the $1$-dimensional moduli space of pseudoholomorphic curves with boundary on $L$ shows that these operations satisfy the graded $A_\infty$ relations:
\begin{equation}
  \label{eq:Ainfinity}
  \sum_{\substack{\beta'+\beta''=\beta\\k_1+j+k_2=k}} (-1)^{\clubsuit_{\underline \x, k_1}}m^{k_1+k_2+1}_{\beta'} (\id^{\otimes k_1} \otimes m^j_{\beta'}\otimes \id^{\otimes k_2})(\underline \x)=0
\end{equation}
where $\underline \x = \x_1\otimes \cdots \otimes\x_k$, and $\clubsuit(\underline \x, k_1)= k_1+\sum_{i=1}^k \deg(\x_i)$.
There is, additionally, an action of $\ZZ/k\ZZ$ on $\uM_{k, \beta}(X, L, J)$ which comes from cyclically relabeling the points and commutes with the evaluation maps. 
This structure presents itself algebraically via a symmetry of the $A_\infty$ products
\begin{equation}
  \label{eq:cyclic}
  \langle m^k_\beta (\x_1\otimes \cdots \otimes \x_k), \x_0 \rangle = (-1)^{(\deg \x_0)\cdot \sum_{i=1}^k \deg \x_k}\langle m^k_\beta (\x_0\otimes \cdots \otimes \x_{k-1}), \x_k \rangle. 
\end{equation}
where $\langle \x_1, \x_2\rangle = \int_L \x_1\wedge \x_2$ is the Poincar\'{e} pairing on de Rham chains. The discrepancy in sign compared to \cite{fukaya2010cyclic} arises from a shift in the degree by 1 (see \cite[Equation 1.1]{fukaya2010cyclic}).

Finally, there is a ``forgetting point'' relation between the moduli spaces of pseudoholomorphic curves: for each $0 < i \leq k$ there is a forgetful map $\upi_{k, i}: \uM_{k, \beta}(X, L, J)\to \uM_{{k-1}, \beta}(X, L, J)$ such that the following diagram commutes for all $0< j\neq i\leq k$:
\[
  \begin{tikzcd}
    \mathcal \uM_{k, \beta} \arrow{r}{\upi_{k, i}} \arrow{dr}{\uev_j} &  \mathcal \uM_{k-1, \beta}\arrow{d}{\uev_{\eps_{ij}}}\\
     & L
  \end{tikzcd}
\]
where $\eps_{ij}=j$ when $i\leq j$ and $j-1$ otherwise. 
The fiber of the forgetful map is $\upi^{-1}_{k, i}(u)=(\theta_i^-, \theta_i^+)$, and $\theta_i^\pm$ are the $\partial D^2$-coordinate values of the point immediately left or right of the $i$-the marked point. By integrating along the $(\theta_i^-, \theta_i^+)$ coordinate and applying Stokes' theorem, one can prove the divisor axiom:  given $\x\in H^1(L)$,
\begin{equation}
  \label{eq:divisor}
  \sum_{i=0}^k(-1)^{(\deg \x)\cdot \sum_{j=1}^i \deg \b_j} m^{k+1}_\beta (\b_1\otimes \cdots \otimes \b_{i}\otimes \x \otimes \b_{i+1}\otimes \cdots \otimes \b_k) = \x(\partial \beta) m^k_\beta(\b_1\otimes \cdots \otimes \b_k).
\end{equation}

\subsubsection{Summary of Construction of Product Structure}
\label{subsubsec:fuk10}
The reason that this only provides a first approximation of the construction is that the space $\uM_{k, \beta}(X, L, J)$ is rarely a smooth manifold, and even when it is, the evaluation map $\uev_0$ is rarely a submersion. These two issues are handled in \cite{fukaya2010cyclic} using abstract perturbation theory to upgrade the underlined constructions from the previous discussion.

\begin{theorem}[\cite{fukaya2010cyclic}]
  Let $L\subset X$ be a graded spin Lagrangian submanifold and $J: TX\to TX$ a choice of compatible almost complex structure. There exists a choice of abstract perturbation datum $\mathcal P$ defining operations $m^k_\beta: \CF(L, J, \mathcal P)^{\otimes k}\to \CF(L, J, \mathcal P)$ satisfying \cref{eq:emptyVanishing,eq:Ainfinity,eq:cyclic,eq:divisor}. The structure is independent of $J$ and additional choices $\mathcal P$ up to pseudoisotopy of filtered $A_\infty$ algebra. As a consequence, $\CF(L, J, \mathcal P)$ and $\CF(L, J', \mathcal P')$ are filtered $A_\infty$ isomorphic.
\end{theorem}
 We include a summary of the proof from \cite{fukaya2010lagrangian} which will allow us to introduce the components that go into the definition of $m^k_\beta$. As a replacement for smoothness of the moduli space $\uM_{k, \beta}(X, L, J)$, \cite[Section 3]{fukaya2010cyclic} exhibits choices of Kuranishi structure for every $k, \beta$ (we will denote the collection of data by $\mathcal P=\{\mathcal P_{k, \beta}\}$). Since $\uev_0$ is not a submersion, one replaces $\uM_{k, \beta}(X, L, J)$ with a thickening by a high-dimensional family of perturbations so that the resulting map $\ev_0$ from the thickened moduli space is a submersion; to correct for the increase of dimension of the domain we subsequently average over those perturbations. The result is a construction for the $m^k_\beta$ which retains many of the nice properties suggested by the placeholder definition \cref{eq:firstApproximation}, including \cref{eq:emptyVanishing}.
To obtain the relations \cref{eq:Ainfinity,eq:cyclic,eq:divisor} the Kuranishi structures and choices of perturbations are chosen compatibly with these forgetful maps and boundary relations.

More precisely: a Kuranishi structure assigns a Kuranishi neighborhood $(V_u, E_u, \Gamma_u, \phi_u, s_u)$ to each $u\in \mathcal \uM_{k, \beta}(X, L, J)$ where $E_u\to V_u$ is a finite dimensional vector bundle, $s_u: V_u\to E_u$ is a section, and $\phi_u: s_u^{-1}(0)\to \uM_{k, \beta}(X, L, J)$ is a homeomorphism to an open subset containing $u$. $\Gamma_u$ is a finite group acting on $E_u$ for which the vector bundle and sections are equivariant. The precise axioms are given in \cite[Section 2]{fukaya2010cyclic}, but informally the axioms state that $\uM_{k, \beta}(X, L, J)$ is locally modeled by a (possibly non-regular, equivariant) zero set. 
From a Kuranishi structure, one constructs a good coordinate system $\{(V_\aa, E_\aa, \Gamma_\aa, \phi_\aa, s_\aa)\}_{\aa\in \mathfrak U}$ for $\mathcal P_{k, \beta}$.  Here, $\mathfrak U$ is a poset indexing the charts. One can perform geometric constructions by replacing $\uM_{k, \beta}(X, L, J)$ with the zero locus of a regular section of $E_\aa\to V_\aa$. 
Because regular sections $s_\aa: V_\aa \to E_\aa$ may not exist (due to the necessity that sections be $\Gamma_\aa$-equivariant) we use multisections instead.  This requires further refining the cover of $V_\aa=\bigcup_{\ii\in \II_\alpha} U_{\aa, \ii}$, taking sections $s_{\aa, \ii}: U_{\aa,\ii}\to E_\aa|_{U_{\aa, \beta}}$, and picking a partition of unity $\chi_{\aa, \ii}: V_\aa\to \RR$ subordinate to the $U_{\aa, \ii}$. We let $\ell_{\aa, \ii}$ denote the degree of multiplicity of the multisection. We promote the maps $\uev_i, \upi_{k,j}$ to maps $\ev_i, \pi_{k,j}$ defined on the Kuranishi charts. 

In earlier work, one would use $s_{\aa, \ii}^{-1}(0)$ as a replacement for the moduli space $\uM_{k, \beta}(X, L, J)$.  However, since the restriction of evaluation maps $\ev_0: s_{\aa, \ii}^{-1}(0)\to L$ may not be a submersion, we ``thicken up'' the entire moduli space. This requires picking a parametrizing space $W_\aa$ for each $\aa$ and taking continuous multi-sections $\mathfrak s_\aa: W_\aa\times V_\aa \to W_\aa \times E_\aa$, satisfying \cite[Definition 4.1]{fukaya2010cyclic}. To account for the fact that $\mathfrak s_\aa^{-1}(0)$ has the wrong dimension (compared to $s_\aa^{-1}(0)$,) we choose volume forms $\rho_\aa$ for the $W_\aa$  to obtain the pair $(\mathfrak s_\aa^{-1}(0), \rho_\aa)$; $\rho_\aa$-averaged operations (like fiberwise integration) will have the correct degree. The continuous multisections can be chosen so that $s_\aa^{-1}(0)=\emptyset$ whenever $\uM_{k, \beta}(X, L, J)=\emptyset$.

Finally, the construction of Kuranishi charts and choices of perturbations (given by continuous multisections) need to be made compatible with the forgetful maps and $A_\infty$ boundary relations in the sense of \cite[Definition 3.1]{fukaya2010cyclic}. Included in the choice of Kuranishi structure and continuous multisections compatible with the forgetful and evaluation maps is a map of posets between $\mathfrak U_{k+1, \beta}$ and $\mathfrak U_{k,\beta}$ which sends the index $\tilde \aa\to \aa$; compatibility between the forms $\rho_{\tilde \aa}, \rho_\aa$ and choices of partitions of unity.

With this structure defined, we obtain from \cite[Eq. 4.8 and 7.1]{fukaya2010cyclic} a product on de Rham forms
\begin{equation}
  \label{eq:mdefinition}
  m^k_\beta (\x_1\tensor\cdots \tensor \x_k):= \sum_{\aa\in \mathfrak U}  \sum_{\ii\in \II_\aa} \sum_{\jj=1}^{\ell_\ii}\left. (\ev_0)_*\left( \frac{1}{\ell_\ii} \chi_\aa\chi_{\ii, \aa}\left(\bigwedge_{j=0}^k \ev_{j}^*\x_j\right)\right) \wedge \rho_\aa\right|_{\mathfrak{s}_{\aa, \ii,\jj}^{-1}(0)} \end{equation}
where the evaluation functions $\ev_{j}: U_{\aa,\ii}\to L$ agree with $\underline \ev_{j}$  upon restriction to $\uM_{k, \beta}(X, L, J)$. The upshot of \cref{eq:mdefinition} is that the set $\mathfrak{s}_{\aa,\ii,\jj}^{-1}(0)$ is a smooth manifold of sufficiently high dimension so that the restricted evaluation map $\ev_0|_{\mathfrak{s}_{\aa,\ii,\jj}^{-1}(0)}$ remains a submersion. Observe that \cref{eq:mdefinition} agrees with \cref{eq:firstApproximation} except that we are pulling/pushing to $\mathfrak{s}_{\aa, \ii, \jj}^{-1}(0)$ as a proxy for $\uM_{k, \beta}(X, L, J)$, and that there is some reweighing so that the terms corresponding to different charts glue together. 

These product terms depend on choices beyond the Lagrangian $L$ and almost complex structure $J$. We will denote the choices made in the abstract perturbation datum by $\mathcal P$, so that $\CF(L, J, \mathcal P)$ is a filtered $A_\infty$ algebra. A fundamental result of \cite{fukaya2010cyclic} is that different choices of $J$ and $\mathcal P$ lead to filtered $A_\infty$ isomorphic $A_\infty$ algebras. 
For this reason, when the choice of perturbation datum is unimportant, we will simply write $\CF(L, J)$.

\subsection{\texorpdfstring{Divisor axiom for $m^0$ in dimension 3}{Divisor axiom for the curvature term in dimension 3}}
\label{subsec:divisorAxiom}
In the setting of dimension three, we can write a divisor axiom for $m^0$ as well.
\begin{lemma}
  \label{lem:zeropairing}
  Let $\dim(L)=3$ and suppose that $\x\in H^1(L)$ satisfies $\x(\partial \beta)=0$. For any $\b\in \CF(L)$ we have 
  \begin{equation} 
   \langle m^k_\beta(\b, \ldots, \b), \x\rangle =0.
  \end{equation}
 \end{lemma}
 \begin{proof}
  We break into cases. 
  When $\beta\neq 0, k>0$, this follows from the cyclic and divisor axiom, as 
  \begin{align*}
  \langle m^k_\beta( \b\tensor \cdots \tensor \b), \x\rangle =& \frac{1}{k}\sum_{i=1}^k (-1)^{i}
  \langle m^k_\beta (\b^{\tensor i} \tensor \x \tensor \b^{\tensor {k-i-1}}), \b \rangle \\
  =& \frac{1}{k}\langle \x(\partial \beta) m^{k-1}(\b, \ldots, \b), \b\rangle =0. 
  \end{align*}
  When $\beta=0$ and $k\geq 1$
  \begin{align*}
  \langle d(\b), \x \rangle =\langle \b, d(\x)\rangle =0 && \langle \b\wedge \b , \x\rangle=0 && \langle m^k_0(\b^{\tensor k}), \x\rangle =0.
  \end{align*}
  This leaves only the case $\beta\neq 0, k=0$. 
  For this case we use a minor modification of \cref{eq:divisor} following the argument given in \cite[Lemma 13.1]{fukaya2010cyclic}.
  Let $\underline \pi_{\beta, 0}: \mathcal M_{\beta,  0}\to \mathcal M_{\beta,  -1}$ be the forgetful map to the moduli space of unmarked disks. As in the discussion following \cite[Lemma 13.1]{fukaya2010cyclic}, this gives us a $S^1$ bundle  $\pi_{\beta, 0}: \mathfrak s_{\tilde \aa, \ii, \jj}^{-1}(0)\to \mathfrak s_{\aa, \ii, \jj}^{-1}(0)$. We then compute:
   \begin{align*}
     \langle m^k_\beta,   \x \rangle 
     =&\int_L \x \wedge \sum_{\tilde \aa\in \mathfrak U_{\beta,  0}}  \sum_{\ii\in I_\aa} \sum_{\jj=1}^{\ell_\ii}\left. (\ev_0)_* \frac{(-1)^{(j-1)\deg \x} \chi_{\tilde \aa}\chi_{\ii, \tilde \aa} }{\ell_\ii}    \wedge \rho_{\tilde \aa} \right|_{\mathfrak{s}_{\tilde \aa, \ii, \jj}^{-1}(0)}\\
     =&   \sum_{\tilde \aa\in \mathfrak U_{\beta,  0}}  \sum_{\ii\in I_\aa} \sum_{\jj=1}^{\ell_\ii} \int_{\mathfrak{s}_{\tilde \aa, \ii, \jj}^{-1}(0)}\left. (\ev_0)^*\x \wedge \frac{(-1)^{(j-1)\deg \x} \chi_{\tilde \aa}\chi_{\ii, \tilde \aa} }{\ell_\ii}    \wedge \rho_{\tilde \aa} \right|_{\mathfrak{s}_{\tilde \aa, \ii, \jj}^{-1}(0)}
     \intertext{Factoring integration over $\mathfrak{s}_{\tilde \aa, \ii, \jj}^{-1}(0)$ through $\pi_{\beta, 0}:\mathfrak{s}_{\tilde \aa, \ii, \jj}^{-1}(0)\to \mathfrak{s}_{\aa, \ii, \jj}^{-1}(0)$, and using the compatibilities $\chi_{\tilde a}=\pi_{\beta, 0}^* \chi_{ a}, \rho_{\tilde a}=\pi_{\beta, 0}^* \rho_a$}
     =&  \sum_{\tilde \aa\in \mathfrak U_{\beta,  0}}  \sum_{\ii\in I_\aa} \sum_{\jj=1}^{\ell_\ii} \int_{\mathfrak{s}_{\aa, \ii, \jj}^{-1}(0)}  \left.(\pi_{\beta, 0})_*((\ev_0)^*\x) \wedge \frac{(-1)^{(j-1)\deg \x} \chi_{\tilde \aa}\chi_{\ii, \tilde \aa} }{\ell_\ii}    \wedge \rho_{\tilde \aa} \right|_{\mathfrak{s}_{\tilde \aa, \ii, \jj}^{-1}(0)}
   \end{align*} 
   This vanishes as $ (\pi_{\beta, 0})_*((\ev_0)^*\x)=\x(\beta)=0$. 
 \end{proof}
\subsection{An unobstructedness criterion}
\label{subsec:unobstructednessCriterion}
We now introduce an unobstructedness criterion utilizing \cref{lem:zeropairing}.
First, let us recall some terminology for filtered $A_\infty$ algebras. We say that a filtered $A_\infty$ algebra $\mathcal A$ is tautologically unobstructed if $m^0=0$. Given $\b\in \mathcal A$, we define the $\b$-deformed $A_\infty$ structure to be 
\[m^k_{\beta, \b}:= \sum_j\sum_{i_0, \ldots, i_k=j} m^{k+j}(\b^{\tensor i_0}\tensor \id \tensor \b^{\tensor i_2}\tensor \cdots \tensor \id \tensor \b^{\tensor i_k}),\] and we write $m^k_\b:=\sum_{\beta}m^k_{\beta, \b}$. Given $E\in \RR_{>0}$, we say that $\b\in \CF(L)$ is an $E$- bounding cochain if $m^0_{\b}\equiv 0\mod T^E$; equivalently that $\b$ solves the Maurer-Cartan equation
\begin{equation}
  \label{eq:maurerCartan}
  0\equiv \sum_k m^k(\b^{\tensor k})\mod T^E.
\end{equation}
We say that $\mathcal A$ is $E$-unobstructed if it has an $E$-bounding cochain. Call $\b$ is a bounding cochain if \cref{eq:maurerCartan} holds for all $E$ and   $\mathcal A$  unobstructed if it has a bounding cochain. Finally, we will say that a Lagrangian submanifold $L$ is $E$-unobstructed or unobstructed if $\CF(L, J)$ is $E$-unobstructed or unobstructed. 

To construct a bounding cochain it is useful to have some constraints on the homology classes of $J$-holomorphic disks.
\begin{definition}
  \label{def:loopGrading}
  Let $L$ be a Lagrangian submanifold. Let $J$ be an almost complex structure. The loop grading group $\Grd(L, J)$ be the subgroup of $H_1(L)$ generated by $J$-pseudoholomorphic disks so that 
  \[ \{ [\partial u]\st u:(D, \partial D^2)\to (X, L) \text{ is $J$ holomorphic} \} \subset \Grd(L, J ).\]
\end{definition}
The smaller we can make the loop grading group, the better chance we have of constructing a bounding cochain.
\begin{example}
  \label{exm:trivialInclusion}
  Given any Lagrangian submanifold, we at a minimum obtain  $\Grd(L, J)\subset \ker(H_1(L)\to H_1(X))$.   
  For example, when we let $V=\{q\}\in \RR^n$ a point, then $X^A\cong L_V\times Q$ and so $\Grd(L_V, J)=0$.
  In fact, in this case, we have something much stronger, that $H_2(X, L)=0$, so there are no $J$-holomorphic disks and $L_V$ is tautologically unobstructed. Thus, the loop grading group provides an intermediate between this stronger condition and what we usually encounter in practice.
\end{example}
Some structures in symplectic geometry only appear under a limiting set of almost complex structures.  For this reason, it is useful to have a framework for a loop grading group of a limiting set of almost complex structures.
\begin{definition}
  \label{def:filteredLoopGrading}
  Let $L$ be a Lagrangian submanifold. A limit loop grading group associated with a family of almost complex structures $\{J_E\}_{E\in \RR_{>0}}$ is the subgroup $\Grd_\infty(L, \{J_E\})\subset H_1(X, L)$ generated by the boundary classes of $J_E$ disks with energy less than $E$ so that:
  \[\{ [\partial u]\st u:(D^2, \partial D^2)\to (X, L) \text{ is $J_{E}$ holomorphic, $\omega(\partial u)\leq E$}\}\subset \Grd_\infty(L, \{J_E\}).\]
\end{definition}

\begin{proof}[Proof of \cref{thm:unobstructedness}]
  Suppose that we have a sequence of almost complex structure $\{J_{E}\}_{E\in \RR}$ so that $\Grd_{\infty}(L, \{J_E\})=0$. Let us denote the $A_\infty$ structure on $\CF(L, J_E)$ by $m^{k, E}_\beta$.
  By \cref{eq:emptyVanishing} we may choose the perturbations so that $m^{k, E}_\beta=0$ whenever $\omega(\beta)<E$ and $[\partial \beta]\neq 0$.
  Following the basic strategy presented in \cite{fukaya2010lagrangian}, we will construct a sequence of $E$-bounding cochains $\b_E\in CF^1(L, J_0)$ with the property that whenever $E'>E$ we have $b_E\cong b_{E'} \mod T^E$.
  The limit $\b:= \lim_{E\to \infty} \b_E$ will then be a bounding cochain for $\CF(L, J_0)$. We start with $\b_0=0$.

  Suppose that we've constructed $\b_E\in \CF(L, J_0)$. Because the Fukaya algebra $\CF(L, J_0)$ is gapped, we can take $E'$ to be the next larger energy level so that $\omega(\beta)>E$ implies $\omega(\beta)\geq E'$. 
  Let $\Phi_{0, E'}: \CF(L, J_0)\to \CF(L, J_{E'})$ be the filtered $A_\infty$ isomorphism arising from the pseudoisotopy. The pushforward $\b'_E:= (\Phi_{0, E'})_*\b_E = \sum_k \Phi_{0, E'}(\b_E^{\tensor k})$ is an $E$-bounding cochain on $\CF(L, J_{E'})$. The lowest-order term of the filtered $A_\infty$ relations tell us \[
    d \circ m^{0,E'}_{\b_{E'}}\equiv 0 \mod T^{E'}.\]  
  We show this is exact by proving for all $\x$ the pairing $ \langle m^{0,E'}_{\b'_E}, \x\rangle\equiv 0 \mod T^{E'}$ vanishes:
  \begin{align*}
    \langle m^{0,E'}_{\b'_E}, \x\rangle =& \sum_{\beta\in H_2(X, L) \st \omega(\beta) \leq E'}T^{\omega(\beta)} \langle m^0_{\beta, \b'_E}, \x\rangle + \text{Terms of order greater than $E'$}\\
    \equiv& \sum_{\beta\in H_2(X, L) \st \omega(\beta) \leq E'}T^{\omega(\beta)} \sum_k\langle m^{k, E'}_\beta((\b'_E)^{\tensor k}), \x\rangle \mod T^{E'}
    \intertext{ By our assumption that $\Grd_\infty(L, \{J_E\})=0$, for every $\beta$ with $\omega(\beta)\leq E'$ and $\x\in H^1(L)$ we have $\x(\partial \beta)=0$. We  apply \cref{lem:zeropairing} to obtain }
    \equiv& 0 \mod T^{E'} \text{ for all } \x\in H^1(L).
  \end{align*}
  Thus $[m^{0,E'}_{\b'_E}]\equiv 0\mod T^{E'}$. Furthermore, $m^{0,E'}_{\b'_E}=0\mod T^E$, so we can find $\y_{E'}$ with the property that $\y_{E'}\equiv 0 \mod T^E$ and $d\y_{E'} = m^0_{\b'_E}\mod T^{E'}$.

  It follows that $\b'_{E}+\y_{E'}$ is an $E'$-bounding cochain on $\CF(L, J_{E'})$. 
  Now let $\Phi_{E', 0}$ be the filtered \text{inverse} to $\Phi_{0, E'}$. We define $\b_{E'}:= (\Phi_{E', 0})_* (\b'_{E}+\y_{E'})$. This is an $E'$-bounding cochain. Since $\Phi_{E', 0}$ is $\Lambda_0$-filtered, 
  \[(\Phi_{E',0})_*(\b'_E+\y_{E'})= \b_E +(\Phi_{E',0})_*\y_{E'}\equiv \b_E\mod T^E.\] 
\end{proof}

\section{Unobstructedness from rigidity}
\label{sec:realizability}
We now apply \cref{thm:unobstructedness} to the setting where $L_V\subset X^A$ is a lift of a tropical subvariety $V\subset Q$. This involves constructing a candidate sequence of almost complex structures so that the limit loop grading group $\Grd_\infty(L, \{J_E\})$ vanishes. 
In \cref{subsec:loopGradingTropical} we use tropicalization of pseudoholomorphic disks to constrain $\Grd_\infty(L_V, \{J_E\})$. 
We subsequently show in \cref{subsec:nondegenerate} that this limiting grading group vanishes whenever the underlying tropical subvariety is rigid. These two results hold in any dimension.
We can then deduce in dimension three that $L_V$ is unobstructed.
\subsection{Constraints on the loop grading of disks}
\label{subsec:loopGradingTropical}
Let $Q$ be an affine manifold with affine metric and let $L_V$ be a Lagrangian lift of a tropical subvariety $V\subset Q$.  
For $\eps$ sufficiently small we have a homotopy equivalence between $B_\eps(|V|)$ --- a small thickening of $|V|$  --- and $|V|$ itself. Composing this with the map $L_V\hookrightarrow \syza^{-1}(B_\eps(|V|))$ gives us a map on first cohomology that we will denote $(\pi_V)_*:H_\bullet(L_V)\to H_\bullet(|V|)$. 
\begin{lemma}\label{lem:trivialRestrictiontoVertex}
  Let $Q, V, L_V$ as above. There exists a sequence of almost complex structures $J_E$ so that 
  \[\Grd_\infty(L_V, \{J_E\})\subset \ker((\pi_V)_*).\]
\end{lemma}

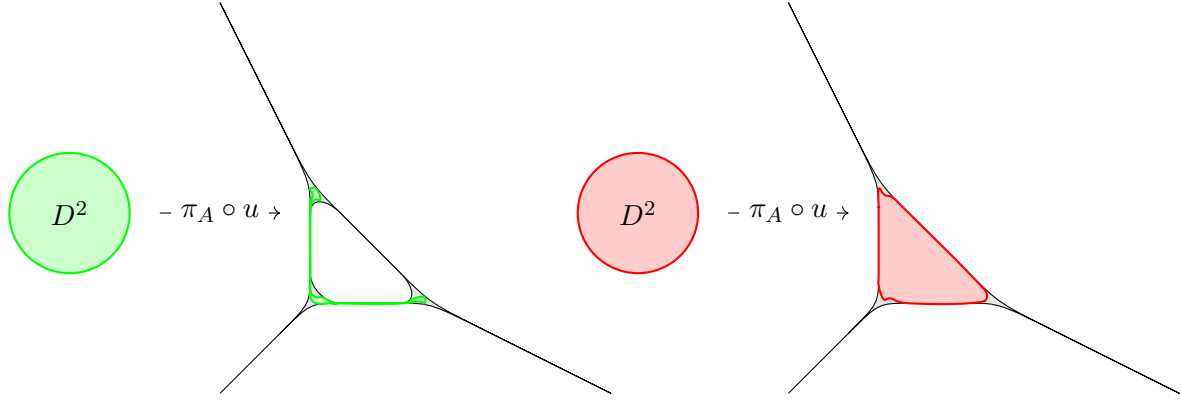
\begin{figure}
\centering
\begin{subfigure}{.45\linewidth}
  \begin{tikzpicture}[scale=.8]
\begin{scope}[]
\draw[fill=gray!20] (-1.5,-0.5) .. controls (-1,0) and (-1,0) .. (-0.5,0.5) .. controls (0,1) and (0,1) .. (0.5,1) .. controls (1,1) and (1,1) .. (1.5,1) .. controls (2,1) and (2,1) .. (3,0.5) .. controls (4,0) and (4,0) .. (5,-0.5) .. controls (4,0) and (4,0) .. (3,0.5) .. controls (2,1) and (2,1) .. (1.5,1.5) .. controls (1,2) and (1,2) .. (0.5,2.5) .. controls (0,3) and (0,3) .. (-0.5,4) .. controls (-1,5) and (-1,5) .. (-1.5,6) .. controls (-1,5) and (-1,5) .. (-0.5,4) .. controls (0,3) and (0,3) .. (0,2.5) .. controls (0,2) and (0,2) .. (0,1.5) .. controls (0,1) and (0,1) .. (-0.5,0.5);
\draw[fill=white] (0,1.5) .. controls (0,1.25) and (0.25,1) .. (0.5,1) node (v1) {} .. controls (1,1) and (1,1) .. (1.5,1) node (v2) {} .. controls (1.75,1) and (1.75,1.25) .. (1.5,1.5) .. controls (1,2) and (1,2) .. (0.5,2.5) .. controls (0.25,2.75) and (0,2.75) .. (0,2.5) .. controls (0,2) and (0,2) .. (0,1.5);
\draw[fill=green!20, draw=green, thick]  (-4,2.5) ellipse (1 and 1);
\draw[fill=green!20, draw=green, thick]  plot[smooth cycle, tension=.7] coordinates {(0,2.9) (0.1,2.9) (0.17,2.73) (0.06,2.71) (0,2.5) (0,1.5) (0,1.2) (0.1,1.1) (0.2,1.1) (0.4,1) (v1) (v2) (1.7,1.07) (1.8,1.1) (1.9,1.1) (1.9,1) (1.8,1.03) (1.5,1) (0.5,1) (0.4,1) (0.1,1) (0,1.1) (0,1.2) (0,1.5) (0,2.5) (0,2.67) (0.04,2.8)};

\end{scope}

\draw[->] (-2.5,2.5) -- node[fill=white]{$\syza\circ u$} (-0.5,2.5);
\node at (-4,2.5) {$D^2$};
\end{tikzpicture}   \caption{$J_E$ pseudoholomorphic disks are allowed to have $\syza$ projection that looks like this.}
\end{subfigure}
\begin{subfigure}{.45\linewidth}
  \begin{tikzpicture}[scale=.8]
\begin{scope}[]
\draw[fill=gray!20] (-1.5,-0.5) .. controls (-1,0) and (-1,0) .. (-0.5,0.5) .. controls (0,1) and (0,1) .. (0.5,1) .. controls (1,1) and (1,1) .. (1.5,1) .. controls (2,1) and (2,1) .. (3,0.5) .. controls (4,0) and (4,0) .. (5,-0.5) .. controls (4,0) and (4,0) .. (3,0.5) .. controls (2,1) and (2,1) .. (1.5,1.5) .. controls (1,2) and (1,2) .. (0.5,2.5) .. controls (0,3) and (0,3) .. (-0.5,4) .. controls (-1,5) and (-1,5) .. (-1.5,6) .. controls (-1,5) and (-1,5) .. (-0.5,4) .. controls (0,3) and (0,3) .. (0,2.5) .. controls (0,2) and (0,2) .. (0,1.5) .. controls (0,1) and (0,1) .. (-0.5,0.5);
\draw[fill=white] (0,1.5) .. controls (0,1.25) and (0.25,1) .. (0.5,1) node (v1) {} .. controls (1,1) and (1,1) .. (1.5,1) node (v2) {} .. controls (1.75,1) and (1.75,1.25) .. (1.5,1.5) node (v3) {} .. controls (1,2) and (1,2) .. (0.5,2.5) .. controls (0.25,2.75) and (0,2.75) .. (0,2.5) .. controls (0,2) and (0,2) .. (0,1.5);
\draw[fill=red!20, draw=red, thick]  (-4,2.5) ellipse (1 and 1);

\end{scope}

\draw[->] (-2.5,2.5) -- node[fill=white]{$\syza\circ u$} (-0.5,2.5);
\node at (-4,2.5) {$D^2$};
\draw[fill=red!20, draw=red, thick, opacity=20]  plot[smooth cycle, tension=.7] coordinates {(0.1,2.8) (0,2.9) (0,2.6) (0,2.5) (0,1.5) (0,1.3) (0.08,1.07) (0.2,1.08) (0.5,1) (v2) (1.8,1.12) (v3) (0.5,2.5) (0.24,2.75)};
\end{tikzpicture}   \caption{$J_E$ pseudoholomorphic disks are not allowed to have $\syza$ projection that looks like this.}
\end{subfigure}
\caption{Pseudoholomorphic disks with boundary on a tropical Lagrangian submanifold must tropicalize in the large complex structure limit.}
\label{fig:diskTropicalization}
\end{figure}
\begin{proof}
  We prove that for any given energy $E$ there exists a choice of almost complex structure $J_E$ so that $\mathcal M_{\beta}(X, L, J_E)$ is empty whenever $(\pi_V)_*[\partial\beta]\neq 0\in H_1(V)$ and $\omega(\beta)<E$. We will prove this by using the tropicalization of pseudoholomorphic disks with boundary on $L_V$ as the complex structure approaches the large fiber limit (a heuristic is drawn in \cref{fig:diskTropicalization}).
 We use the metric to determine a splitting 
 \[TX^A\cong\syza^*TQ\oplus \syza^*T^*Q,\]
 as well as an isomorphism $\Phi: \syza^*T^*Q\to \syza^*TQ$. For $A\in \RR_{>0}$, we define a compatible almost complex structure by 
 \[I_A:=\begin{pmatrix} 0 & A\cdot \Phi \\ - (A\cdot \Phi)^{-1} & 0 \end{pmatrix}: \syza^*TQ\oplus \syza^*T^*Q\to \syza^*TQ\oplus \syza^*T^*Q\]
 Using the metric to determine locally flat orthonormal coordinates $(q_i, p_i)$ on $TX_A$ we obtain
  \begin{align*}
    I_A\partial_{q_i} =- A^{-1}\cdot \partial_{\p_i} && I_A \partial_{\p_i}= A \cdot \partial_{q_i}.
  \end{align*}
   We now show that the symplectic area of an $I_A$-pseudoholomorphic disk is lower bounded by a multiple of the $g$-area of its projection to $Q$:
  \begin{align*}
    \int_{D^2}u^*\omega =& \iint_{s, t} \sum_{i=1}^n dq_i\wedge d\p_i \left(\partial_s u, \partial_t u \right)dsdt = \iint_{s, t}\sum_i\left( \partial_s u_{q_i}\cdot \partial_t u_{\p_i}-\partial_t u_{q_i}\cdot \partial_s u_{\p_i}\right)dsdt\\
    =&\iint_{s,t} A \sum_{i}\left( (\partial_s u_{q_i})^2+(\partial_t u_{q_i})^2\right) dsdt= \frac{A}{n} \iint_{s,t}\sum_{i, j} (\partial_s u_{q_i})^2 + (\partial_t u_{q_j})^2 ds dt\\
    \geq& \frac{A}{n}\iint_{s, t} (\partial_s u_q )\cdot (\partial_t u_q)ds dt = \frac{A}{n}\cdot \Area_{Q,g}(\pi\circ u).
  \end{align*}
  Observe that there exists $C>0$ so that  $\Area_{Q,g(\underline u)}>C$, whenever $\underline u$ represents a nontrivial class in $H_2(Q, B_\eps(|V|))$.

  We are now ready to prove the lemma. Take our sequence of almost complex structures to be $J_E=I_{nE/C}$. For this choice of almost complex structure, we have for any class $\beta$ supporting a holomorphic disk with $\omega(\beta)\leq E$  :
  \[E\geq \omega(\beta) \geq \frac{A}{n}\Area_{Q, g}(\syza\circ u)> E \text{ if $(\syza)_*[\beta]\neq 0 \in H_2(Q,B_\eps(|V|))$}\]
  So we conclude that  $(\syza)_*[\beta]=0$ and in particular $(\syza)_*[\partial \beta]=0\in H_1(B_\eps(|V|))$.
\end{proof}
\begin{remark}
\Cref{lem:trivialRestrictiontoVertex} can be upgraded to prove that for suitable $J$ the boundary of pseudoholomorphic disks must be contractible in $|V|$, not simply null-homologous.
\end{remark}
\begin{remark}  
  In \cite[Figure 3]{hicks2021observations}, it was noted that when $Q$ has an affine structure with discriminant locus and boundary (these examples show up, for instance, for $X^A=\CP^2$), there \emph{cannot} exist a tropicalizing almost complex structure.
\end{remark}

\subsection{First cohomology of Lagrangian lifts of tropical subvarieties}
\label{subsec:nondegenerate}
Let $\dim(V)=k$ and $\dim(X)=n$. 
To check that $\Grd_\infty(L_V, \{J_E\})=0$ by looking at $V$, we need a good description of the cohomology of $L_V$.  
The broadest class where we have a clear description of the cohomology of $L_V$ is when $V$ admits a pair-of-pants decomposition.
While we will only need to compute $H^1(L_V)$ when $V$ is a curve, work in the level of generality where $V$ is of arbitrary dimension and admits a pair of pants decomposition as it may be of use in future work.

\subsubsection{Local model: the pair of pants}
The $j-1$-dimensional tropical pair of pants $\Gamma_j\subset \RR^{j}$ is the tropical hypersurface corresponding to the tropical polynomial $1\oplus q_1\oplus \cdots \oplus q_{j}$. This is the tropicalization of a hyperplane inside of $\mathbb{P}^{j}$. We say that $V$ admits a pair-of-pants decomposition if every point has a neighborhood locally modeled on the stabilization of a pair of pants. More precisely: for any stratum $\sigma$ of $V$, we write $V_\sigma:=\{\tau \st \tau>\sigma\}$ for the star of $\sigma$. 
We say that $V$ has a pair-of-pants decomposition if for all $\sigma$ there exists a neighborhood $U_\sigma\subset Q$ so that $|V|\cap U_\sigma = |V_\sigma|\cap U_\sigma$ and an affine map $\phi: U_\sigma\to \RR^{n}$ diffeomorphic onto its image so that $\phi(V_\sigma)= (|\Gamma_j|\times \RR^{k-j}\times \emptyset)\cap {\phi(U_\sigma)}\subset (\RR^{j}\times \RR^{k-j}\times \RR^{n-k})$. It suffices to check these conditions at the stars of the vertices as opposed to all strata.

In \cite{hicks2020tropical,matessi2018lagrangian}, a Lagrangian lift $L_{\Gamma_j}\subset T^*\RR^{j}/T^*_\ZZ\RR^{j}$ of the tropical pair of pants $\Gamma_j\subset \RR^{j}$ was constructed (independently, the lift was constructed for $L_{\Gamma_2}$ by \cite{sheridan2020lagrangian,mak2020tropically,mikhalkin2018examples}).
In \cite{hicks2020tropical} the Lagrangian $L_{\Gamma_j}$ is presented topologically as two copies of $D^j$ identified at $j$ boundary points, from which it follows that 
\[  
  H^i(L_{\Gamma_j})=\left\{\begin{array}{cc} \RR & i=0\\ \RR^{j} & i=1 \\0 & \text{otherwise.}\end{array}\right.
\]
It will be useful to give a coordinate-free presentation of the cohomology of the tropical pair of pants. Returning to our notation from before: let $Q=\RR^{j}$; $N_Q$ be the sheaf of integral tangent fields;  $M_Q\subset T^*Q$ be the dual bundle of integral covector fields; and let $v\in \Gamma_j\subset Q$ be the vertex of the pair of pants. There is a projection from the total space $T^*Q/M_Q$ to the real $j$ dimensional torus $M_v\tensor \RR/M_v$, whose cohomology is canonically identified as the exterior algebra generated on the lattice $N_v$.
The composition $i: L_{\Gamma_j}\to X^A_{\RR^j}\simeq M_v\tensor \RR/M$ induces a map $i^*: N_v\to H^\bullet(L_{\Gamma_j})$. This map is surjective, an isomorphism in degrees 0 and 1, and zero in higher degrees; in particular: 
\begin{align}
  \label{eq:popCohomology}
  H^0(L_{\Gamma_j})\cong \RR, && H^1(L_{\Gamma_j})\cong H^0(V, i^*_v N_Q).
\end{align}
We now consider the cohomology of the local model for a Lagrangian lift of a tropical subvariety with a pair-of-pants decomposition.
The restriction to $L_\sigma := L_V\cap \syza^{-1}(U_\sigma)$ is based on a local model:
\[L_\sigma = \left(L_{\Gamma_j}\times T^{j-k}\times \emptyset \right)\cap \phi_\sigma (\tilde U_\sigma)\]
which sits inside $\RR^j\times \RR^{j-k}\times \times \RR^{n-k-j}$. 
By application of K\"unneth formula and \cref{eq:popCohomology}, we identify the first cohomology of the pieces with
\begin{align*}
  H^0(L_\sigma)\cong \RR &&  H^1(L_\sigma)\cong H^0(V_\sigma,\underline i^*_\sigma N_Q/N_\sigma).
\end{align*}
Furthermore observe that when $\sigma < \tau$ we have a containment $N_\sigma < N_\tau$, and thus a map 
\[H^0(V_\sigma, \underline i^*_\sigma N_Q/N_\sigma)\to H^0(V_\tau, \underline i^*_\tau  N_Q/N_\tau).\]
This agrees with the pullback map $H^1(L_\sigma)\to H^1(L_\tau)$.

\subsection{An efficient cover for computing cohomology}
Recall that we can associate to each stratum $\sigma$ of $V$ the star of the stratum, $V_\sigma$. 
We can similarly form an open set $\tilde U_\sigma:= B_\eps(|V_\sigma|) \setminus \left(\bigcup_{\tau\not \geq \sigma} \bar B_\eps(|\tau|)\right)$. This is a slightly interior open subset of $B_\eps(V_\sigma)$. These sets form an ``efficient'' cover in the following sense:
\begin{proposition}\label{prop:efficientComplex}
  There is an exact sequence of chain complexes
  \begin{equation}
    \label{eq:efficientComplex}
    \bigoplus_{\sigma \st \dim(\sigma)=k} C^\bullet(\tilde U_\sigma)\ot \bigoplus_{\sigma \st \dim(\sigma)=k-1} C^\bullet(\tilde U_\sigma)\ot \cdots \ot \bigoplus_{e\in V} C^\bullet(\tilde U_e)\ot \bigoplus_{v\in V} C^\bullet(\tilde U_v)\ot C^\bullet(B_\eps(V)).
  \end{equation}
\end{proposition}
\begin{proof}
  First, observe in the setting where our polyhedral complex $V$ is simplicial, $\tilde U_\sigma$ agrees with $\bigcap_{v< \sigma}\tilde U_v$, and so in this case \cref{eq:efficientComplex} agrees with the \v{C}ech de-Rham exact sequence associated to the cover $\{L_V\}_{v\in V}$. For general $V$, we can always find $\tilde V$ a polyhedral subdivision. The standard argument for proving that subdivision is a homotopy equivalence of simplicial cochain complexes  (here adapted to CW cohomology with polyhedral cells) proves exactness for \cref{eq:efficientComplex}.
\end{proof}

By the same argument as \cref{prop:efficientComplex}, the cohomology of $L_V$ fits in the exact sequence of cochain complexes:
\[\bigoplus_{\sigma \st \dim(\sigma)=k} C^\bullet(L_\sigma)\ot \bigoplus_{\sigma \st \dim(\sigma)=k-1} C^\bullet(\tilde L_\sigma)\ot \cdots \ot \bigoplus_{e\in V} C^\bullet(L_e)\ot \bigoplus_{v\in V} C^\bullet(L_v)\ot C^\bullet(L_V).\]

\subsection{Rigidity and Unobstructedness}
We are now equipped to compare the first cohomology of $L_V$ with $\syza^{-1}(\tilde U_V)$.
\begin{lemma}
  \label{lem:rigidGrading}
  Suppose that $V\subset Q$ is rigid. Then $H^1(\syza^{-1}(\tilde U_V))\to H^1(L_V)$ is surjective.
\end{lemma}

\begin{proof}
  We compute $H^1(L_V)$ and  $H^1(\syza^{-1}(\tilde U_V))$ using the spectral sequence $E^{p, q}_i\Rightarrow H^\bullet(L_V)$, where
\begin{align*}
  E^{p, q}_0= \bigoplus_{\sigma \st \dim(\sigma)=p} C^q(L_\sigma)&& 
  F^{p, q}_0 = \bigoplus_{\sigma \st \dim(\sigma)=p} C^q(\syza^{-1} (\tilde U_\sigma))
\end{align*} 
  Surjectivity can be computed on graded components of the spectral sequence, so we need to show that the maps $i^{0,1}_\infty: F^{0,1}_\infty\to E^{0,1}_\infty$ and $i^{1,0}_\infty: F^{1, 0}_\infty\to E^{1,0}_\infty$ are surjective. 
  
  \noindent \textbf{Surjectivity of $i^{1,0}_\infty$:} As the incidences of our covers agree and each cover component is connected there is an isomorphism of complexes $(E^{k,0}_1, d^{k, 0}_1)\cong (F^{k, 0}_1, d^{k, 0}_1)$. For degree reasons, $E^{1, 0}_2\cong F^{1, 0}_\infty$ and $F^{1, 0}_2\cong F^{1, 0}_\infty$. It follows that $i^{1, 0}_\infty$ is an isomorphism identifying 
  \[H^1(V)\cong F^{1, 0}_\infty \cong E^{1,0}_\infty.\]

  \noindent\textbf{Surjectivity of $i^{0,1}_\infty$:} By degree reasons we have the diagram with exact rows
    \[
    \begin{tikzcd}
      0\arrow{r} &F^{0, 1}_\infty \arrow{r} \arrow[red]{d}{i^{0, 1}_\infty} & F^{0, 1}_2\arrow{r}{d^{0,1}_2} \arrow[blue]{d}{i^{0,1}_2}& F^{2,0}_2\arrow{d}{\sim}\\
      0 \arrow{r} &E^{0, 1}_\infty \arrow{r}& E^{0, 1}_2\arrow{r}{d^{0,1}_2}& E^{2,0}_2\\
    \end{tikzcd}\]
    The surjectivity of the red arrow will follow from the surjectivity of $i^{0,1}_2$ and the four-lemma.
    Further unwinding the spectral sequence yields the commutative diagram with exact rows
    \[\begin{tikzcd}   
    0\arrow{d} \arrow{r}{} & F^{1, 0}_2 \arrow[blue]{d}{i^{0,1}_2}\arrow{r}{} & F^{1, 0}_1 \arrow{d} \arrow{r}{d^{1,0}_1} & F^{11}_1  \arrow{d}
\\
     0 \arrow{r} & E^{1, 0}_2 \arrow{r} & E^{1, 0}_1 \arrow{r}{d^{1, 0}_1} & E^{11}_1 \\
    \end{tikzcd}\]
    After identifying $E^{p,q}_1=\bigoplus_{\sigma \st \dim(\sigma)=p} H^q( L_\sigma),  F^{p,q}_1\simeq \bigoplus_{\sigma \st \dim(\sigma)=p} H^q(\syza^{-1} (\tilde U_\sigma))$ and making the substitutions from the previous section,
    \begin{align*}
      H^1(L_v)=H^0(v,\underline i^*_vN_Q) && H^1(L_e)=H^0(N_e; i^*_vN_Q/N_e) \\ H^1(\syza^{-1}(\tilde U_v))=H^0(|V_v|,\underline i^*_vN_Q)&&H^1(\syza^{-1}(\tilde U_e))=H^0(||V_e|,\underline i^*_eN_Q)
    \end{align*}
    we obtain the following diagram with exact rows and columns
    \[\begin{tikzpicture}[xscale=1.2 ,yscale=1.5]

      \draw[dashed] (-2,-3) -- (11,1);
      \draw[dashed] (-2,-2) -- (11,2);
      \node[fill=white] (v11) at (8,2) {$ 0$};
      \node[fill=white] (v7) at (4,1) {$0$};
      \node[fill=white] (v12) at (8,1) {$\bigoplus_{v\in  V}H^0(|V_e|, N_e)$};
      \node[fill=white] (v1) at (-2,0) {$0$};
      \node[fill=white] (v3) at (1,0) {$F_2^{1, 0}$};
      \node[fill=white] (v8) at (4,0) {$ \bigoplus_{v\in V} H^0(|V_v|, \underline{i}_v^* N_Q)$};
      \node[fill=white] (v13) at (8,0) {$ \bigoplus_{e\in V} H^0(|V_v|, \underline{i}_e^* N_Q)$};
      \node[fill=white] (v2) at (-2,-1) {$ 0$};
      \node[fill=white] (v4) at (1,-1) {$ E_2^{1,0} $};
      \node[fill=white] (v9) at (4,-1) {$ \bigoplus_{v\in V} H^0(|V_v|, \underline{i}_v^* N_Q)$};
      \node[fill=white] (v14) at (8,-1) {$ \bigoplus_{e\in V} H^0(|V_v|, \underline{i}_e^* N_Q/N_e)$};
      \node[fill=white] (v5) at (1,-2) {$\text{coker}(i_2^{0,1}) $};
      \node[fill=white] (v10) at (4,-2) {$ 0$};
      \node[fill=white] (v6) at (1,-3) {$ 0$};
      \draw[->, blue]  (v3) edge node[midway,right]{$i^{0,1}_2$} (v4);
      \draw[->]  (v4) edge (v5);
      \draw[->]  (v5) edge (v6);
      \draw[->]  (v7) edge (v8);
      \draw[->]  (v8) edge (v9);
      \draw[->]  (v9) edge (v10);
      \draw[->]  (v11) edge (v12);
      \draw[->]  (v12) edge (v13);
      \draw[->]  (v13) edge (v14);
      \draw[->]  (v7) edge (v12);
      \draw[->]  (v1) edge (v3);
      \draw[->]  (v3) edge  (v8);
      \draw[->]  (v8) edge (v13);
      \draw[->]  (v2) edge (v4);
      \draw[->]  (v4) edge (v9);
      \draw[->]  (v9) edge (v14);
      \draw[->]  (v5) edge (v10);
\end{tikzpicture}\]
    We can almost apply the snake lemma to obtain the desired result, but the failure of right-exactness in the second row requires a small modification. By exactness of columns, we see that the total complex associated with the above diagram has cohomology vanishing in degrees corresponding to the two dashed diagonals. If we instead compute the cohomology by spectral sequence first taking the horizontal differential we obtain
    \[\begin{tikzpicture}[xscale=.8]

      \draw[dashed] (-2,-3) -- (11,1);
      \draw[dashed] (-2,-2) -- (11,2);
      \node[fill=white] (v11) at (8,2) {$ 0$};
      \node[fill=white] (v7) at (4,1) {$0$};
      \node[fill=white] (v12) at (8,1) {$\bigoplus_{v\in  V}H^0(|V_e|, N_e)$};
      \node[fill=white] (v1) at (-2,0) {$0$};
      \node[fill=white] (v3) at (1,0) {$0$};
      \node[fill=white] (v8) at (4,0) {$ 0$};
      \node[fill=white] (v13) at (8,0) {$ H^1(V, i^*N_Q)$};
      \node[fill=white] (v2) at (-2,-1) {$ 0$};
      \node[fill=white] (v4) at (1,-1) {$ 0 $};
      \node[fill=white] (v9) at (4,-1) {$ 0$};
      \node[fill=white] (v14) at (8,-1) {$*$};
      \node[fill=white] (v5) at (1,-2) {$\text{coker}(i_2^{0,1}) $};
      \node[fill=white] (v10) at (4,-2) {$ 0$};
      \node[fill=white] (v6) at (1,-3) {$ 0$};
      \draw[->]  (v3) edge  (v4);
      \draw[->]  (v4) edge (v5);
      \draw[->]  (v5) edge (v6);
      \draw[->]  (v7) edge (v8);
      \draw[->]  (v8) edge (v9);
      \draw[->]  (v9) edge (v10);
      \draw[->]  (v11) edge (v12);
      \draw[->]  (v12) edge (v13);
      \draw[->]  (v13) edge (v14);
\end{tikzpicture}\]
    The second and third page agree:
      
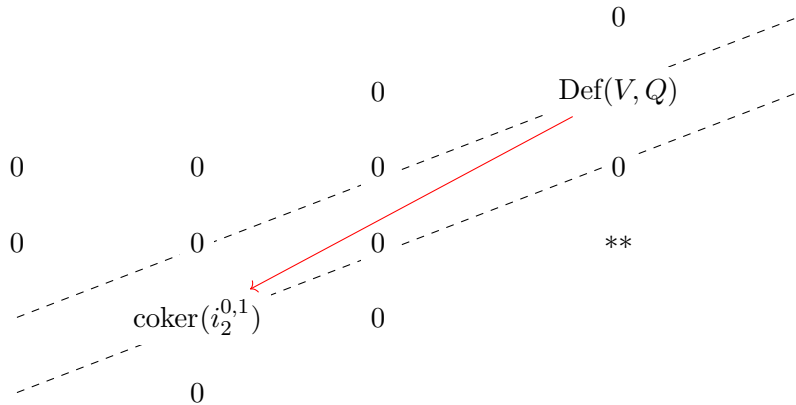
\[\begin{tikzpicture}[xscale=.8]

        \draw[dashed] (-2,-3) -- (11,1);
        \draw[dashed] (-2,-2) -- (11,2);
        \node[fill=white] (v11) at (8,2) {$ 0$};
        \node[fill=white] (v7) at (4,1) {$0$};
        \node[fill=white] (v12) at (8,1) {$\text{Def}(V, Q)$};
        \node[fill=white] (v1) at (-2,0) {$0$};
        \node[fill=white] (v3) at (1,0) {$0$};
        \node[fill=white] (v8) at (4,0) {$ 0$};
        \node[fill=white] (v13) at (8,0) {$ 0$};
        \node[fill=white] (v2) at (-2,-1) {$ 0$};
        \node[fill=white] (v4) at (1,-1) {$ 0 $};
        \node[fill=white] (v9) at (4,-1) {$ 0$};
        \node[fill=white] (v14) at (8,-1) {$**$};
        \node[fill=white] (v5) at (1,-2) {$\text{coker}(i_2^{0,1}) $};
        \node[fill=white] (v10) at (4,-2) {$ 0$};
        \node[fill=white] (v6) at (1,-3) {$ 0$};
  \draw[red,->]  (v12) edge (v5);
  \end{tikzpicture}
        \]
        where the $(*), (**)$ represent some possibly non-zero group that we need not compute.
As the cohomology along the dashed lines must vanish, the differential on the 3rd page of the spectral sequence (drawn in red) provides an isomorphism $\Def(V, Q)\cong \coker(i^{0, 1}_2)$, finishing the proof of the lemma.
\end{proof}

\begin{proof}[Proof of \cref{thm:realizable}]
On homology, \cref{lem:rigidGrading} tells us that 
\[H_1(L_V)\to H_1(\syza^{-1}(U_V))\cong H_1(F)\oplus H_1(V)\]
is injective. It follows that the limiting grading group $\Grd_\infty(L, \{J_E\})$ constructed in \cref{lem:trivialRestrictiontoVertex} is trivial.
Unobstructedness of $L_V$ is then a consequence of \cref{thm:unobstructedness}.
\end{proof}
\section{Examples}
\label{sec:examples}
We provide some examples and non-examples of rigid tropical curves.

\begin{example}
  Let $V_1, V_2\subset Q$ be two tropical hypersurfaces in a tropical threefold, whose stabilizer under translations is trivial. Suppose that $V:=V_1\cap V_2$ admits a pair-of-pants decomposition. Then $V$ is not rigid (as the dimension of the deformation space is at least $\dim(Q)$).
  This shows that rigid tropical curves in $Q$ provide a new class of tropical curves whose realizability is not covered by previous constructions.
\end{example}
\begin{example}
  \label{ex:rigid1}
  We provide an example in dimension two.
  Consider the tropical curve in a tropical abelian variety drawn in \cref{fig:rigid1}, which has edges pointing in directions 
  \begin{align*}
  e_1: \langle 1, 0\rangle, && e_2: \langle 0,1 \rangle, && e_3:\langle -1, -1\rangle
  \end{align*} 
  living in the torus $\RR^n/(\langle 1, 3\rangle, \langle 3, 1 \rangle)$. 
  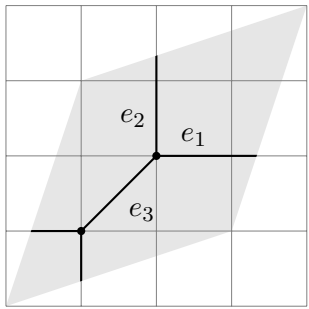
\begin{figure}
  \centering 
  \begin{tikzpicture}

    \fill[gray!20] (0,3) node (v1) {} -- (-1,0) -- (-4,-1) -- (-3,2) -- cycle;

    \draw [help lines, step=1cm] (-4,-1) grid (v1);

    \clip (0,3)  -- (-1,0) -- (-4,-1) -- (-3,2) -- cycle;
    \draw[thick] (-3,-1) -- (-3,0) -- (-4,0);

    \draw[thick] (-3,0) -- (-2,1) -- (-2,2.5) (-0.5,1) -- (-2,1);
    \node[above] at (-1.5,1) {$e_1$};
\node[left] at (-2,1.5) {$e_2$};
\node[below right] at (-2.5,0.5) {$e_3$};
\node[fill=black, circle, scale=.3] at (-3,0) {};
\node[fill=black, circle, scale=.3] at (-2,1) {};
\end{tikzpicture}   \caption{A rigid tropical curve in a tropical abelian surface.}
  \label{fig:rigid1}
  \end{figure}
  Let $c_1=e_1+e_3, c_2=e_2+e_3$ be the horizontal and vertical cycles.
  Give $H_1(|V|, N_Q)$ the basis $c_1^1, c_1^2, c_2^1, c_2^2$ (corresponding to the constant vectors $\langle 1, 0\rangle, \langle 0, 1\rangle \in N_Q$ along the cycles $c_1, c_2$). We give $\RR^E$ the basis $e_1, e_2, e_3$. In this basis, the map $\phi$ is given by the matrix
  \[
  \begin{pmatrix}
  1 & 0 & 1\\
  0 & 1 & 1\\\hline
  0 & 0 & 1\\
  0 & 1 & 1
  \end{pmatrix}: \RR^E\to H_1(|V|, \underline i_V^*N_Q)
  \]
  where the horizontal line divides the target basis into the $c_1$ and $c_2$ components. This matrix has rank 3 so $\phi$ is injective and $V$ is rigid.
  In fact, in dimension two, all examples have the same combinatorial type as this one, as \cref{eq:superabundanceInequality} states that $\Def(V, Q)\neq 0$ as soon as $V$ has genus at least 3.
\end{example}
\begin{example}
Unlike superabundance, rigidity is preserved under stabilization. 
  When we stabilize the curve from \cref{ex:rigid1} to the curve $V\times q \subset Q\times S^1$, the resulting tropical curve (pictured in \cref{fig:stabilization}) remains rigid, and therefore realizable. 
  \begin{figure}
    \centering
    \begin{tikzpicture}

\begin{scope}[shift={(1.5,-1.5)}]

\draw[gray] (0,3) node (v1) {} -- (-1,0) -- (-4,-1) -- (-3,2) -- cycle;
\end{scope}
\draw[gray] (-5.5,0.5) -- (-2.5,-2.5)   (-1.5,4.5) -- (1.5,1.5) (-4.5,3.5) -- (-1.5,0.5);
\begin{scope}[]

\fill[gray!20] (0,3) node (v1) {} -- (-1,0) -- (-4,-1) -- (-3,2) -- cycle;
\end{scope}
\draw (-2.5,-0.5) -- (-2.5,0.5) -- (-3.5,0.5);
\draw (-2.5,0.5) -- (-1.5,1.5) -- (-1.5,2.5) (-0.5,1.5) -- (-1.5,1.5);

\begin{scope}[shift={(-1.5,1.5)}]

\draw[gray] (0,3) node (v1) {} -- (-1,0) -- (-4,-1) -- (-3,2) -- cycle;
\end{scope}
\draw[gray](-2.5,1.5) -- (0.5,-1.5);

\node[fill=black, circle, scale=.3] at (-2.5,0.5) {};
\node[fill=black, circle, scale=.3] at (-1.5,1.5) {};
\end{tikzpicture}     \caption{Stabilization preserves rigidity of tropical curves.}
    \label{fig:stabilization}
  \end{figure}
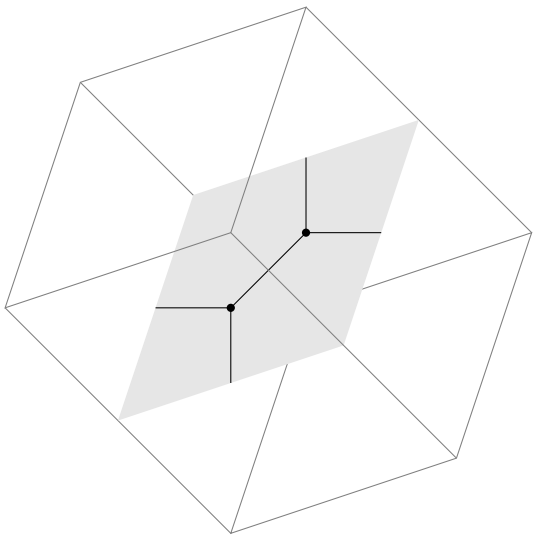
\end{example}
\begin{example}
We give an example of a rigid tropical curve in a tropical abelian threefold that does not arise from stabilization.
Consider the tropical curve drawn in \cref{fig:rigid2} with edges of primitive length in directions
\begin{align*}
  e_1 : \langle 1, 0, 0 \rangle, && e_2':\langle 0, 1 , 0 \rangle,  && e_3:\langle 0, 0, -1 \rangle, \\
  e_1' : \langle -1, 0, 0 \rangle, && e_2 :\langle 0, -1, 0 \rangle, && e_3' : \langle 0, 0, 1 \rangle, \\
  e_{12}: \langle 1, 1, 0 \rangle, && e_{13} : \langle 1, 0, 1 \rangle, && e_{23} : \langle 0, 1, 1\rangle.
\end{align*}

\begin{figure}
  \centering
  \begin{tikzpicture}

\begin{scope}[shift={(1,-1.5)}, dotted]

\draw (2.5,-2.5) -- (-1.5,1.5);

\fill[red!20] (-1,1.5) -- (0,0.5) -- (1,1) -- (0,2) -- cycle;
\fill[green!20] (-1,1.5) -- (0,0.5) -- (0,-1) -- (-1,0) -- cycle;
\fill[blue!20] (0,0.5) -- (0,-1) -- (1,-0.5) -- (1,1) -- cycle;

\draw (-1.5,0.5) -- (2.5,2.5);
\draw (2.5,-2.5) -- (2.5,2.5);
\draw (2.5,-2.5) -- (-1.5,-4.5) -- (-1.5,0.5) -- (-5.5,4.5) -- (-1.5,6.5) -- (2.5,2.5);
\draw (-1.5,6.5) -- (-1.5,1.5) -- (-5.5,-0.5) -- (-1.5,-4.5) (-5.5,-0.5) -- (-5.5,4.5);
\end{scope}

\draw (0,0) -- node[midway, above]{$e_1$}(1,0.5) --node[midway, above]{$e_2$}  (2,-0.5) --node[midway, right]{$e_3$} (2,-2) --node[midway, below]{$e'_1$} (1,-2.5) -- node[midway, below]{$e'_2$}(0,-1.5) -- node[midway, left]{$e'_3$}(0,0);
\draw[->] (0,0) -- (-1,1)node[above]{$e_{13}$};
\draw[->>] (0,-1.5) -- (-1,-2)node[above]{$e_{23}$};
\draw[->>>] (1,-2.5) -- (1,-3.5)node[right]{$e_{12}$};
\draw[-<] (2,-2) -- (3,-3);
\draw[-<<] (2,-0.5) -- (3,0);
\draw[-<<<] (1,0.5) -- (1,1.5);
\node at (-2.5,-4) {$\langle 3, 3, 1\rangle $};

\node at (-4.5,0.5) {$\langle 3, 1, 3\rangle$};

\node at (1.5,-5) {$\langle 1, 3, 3\rangle $};
\begin{scope}[help lines, shift={(2,-1)}]
\node at (-4.5,-6) {$\langle 0, 1, 0 \rangle $};
\node at (-6,-2.5) {$\langle 0, 0, 1 \rangle $};
\node at (-7.5,-5) {$\langle 1, 0, 0 \rangle $};
\draw[->] (-6,-4.5) -- (-7,-5);
\draw[->] (-6,-4.5) -- (-5,-5.5);
\draw[->] (-6,-4.5) -- (-6,-3);
\end{scope}
\node[fill, circle, scale=.3] at (0,0) {};
\node[fill, circle, scale=.3] at (0,-1.5) {};
\node[fill, circle, scale=.3] at (1,-2.5) {};
\node[fill, circle, scale=.3] at (2,-2) {};
\node[fill, circle, scale=.3] at (2,-0.5) {};
\node[fill, circle, scale=.3] at (1,0.5) {};
\end{tikzpicture}   \caption{A tropical curve in a tropical abelian torus. The RGB cube is the unit cube drawn for reference. The dotted lines denote the boundary of the fundamental domain of the abelian torus. }
  \label{fig:rigid2}
  \end{figure}
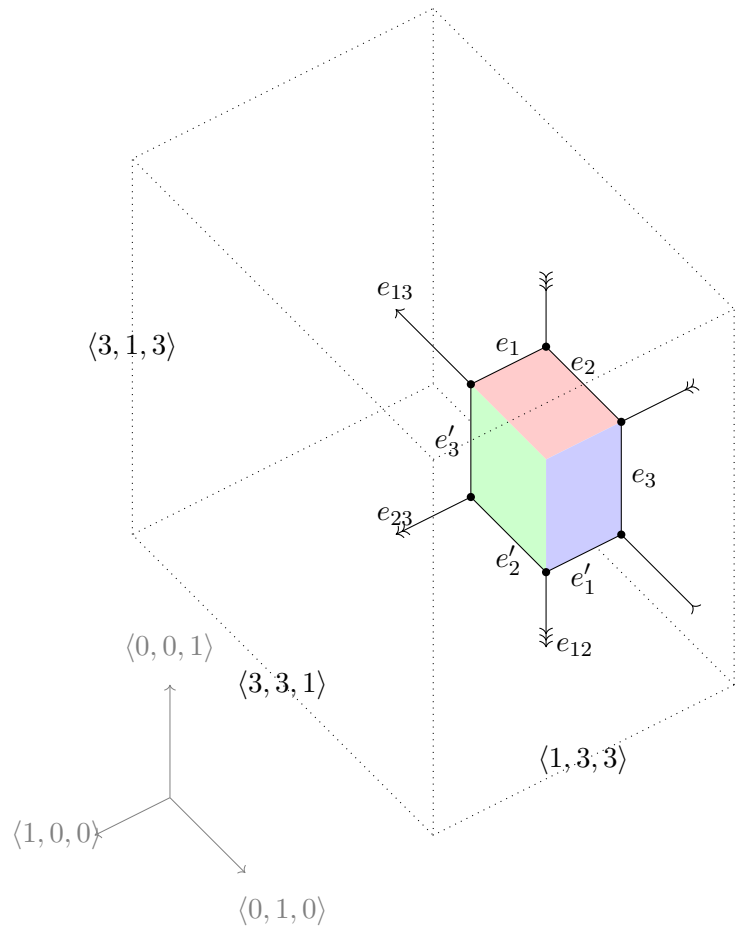

We give $\RR^E$ the basis $e_1, e_2, e_3, e_1', e_2', e_3', e_{12},e_{13},e_{23}$
each of integral length one. This closes up to a closed tropical curve when drawn within a fundamental domain with edges in directions $\langle 3, 3, 1 \rangle, \langle 3, 1, 3 \rangle, \langle 1, 3, 3\rangle$.  The following four cycles provide a basis for $H^1(|V|)$
\begin{align*} c_{12}=e_{12}+e_1'+e_3'+e_2', && c_{13}=e_{13}+e_1'+e_2+e_3', && c_{23}= e_{23}+e_2'+e_1+e_3', \\
  && a=e_1+e_2'+e_3+e_1'+e_2+ e_3'
\end{align*}

A basis for $H^1(|V|, \underline i_V^*N_Q)$ is then given by $c_{12}^1, c_{12}^2, c_{12}^3, c_{13}^1, c_{13}^2, c_{13}^3, c_{23}^1, c_{23}^2, c_{23}^3, a^1, a^2, a^3$, where the superscripts denote the standard basis directions of $N_Q$. In this basis, the map $\phi: \RR^E\to H^1(|V|, \underline i_V^*N_Q)$ is given by the matrix
\[\begin{pmatrix}
0&0&0&1&0&0&1&0&0\\
0&0&0&0&1&0&1&0&0\\
0&0&1&0&0&0&0&0&0\\\hline
0&0&0&1&0&0&0&1&0\\
0&1&0&0&0&0&0&0&0\\
0&0&0&0&0&1&0&1&0\\\hline
1&0&0&0&0&0&0&0&0\\
0&0&0&0&1&0&0&0&1\\
0&0&0&0&0&1&0&0&1\\\hline
1&0&0&1&0&0&0&0&0\\
0&1&0&0&1&0&0&0&0\\
0&0&1&0&0&1&0&0&0
\end{pmatrix}\]
where the horizontal lines again divide the target into the $c_{12},c_{13}, c_{23}$ and $a$-cycle components. The matrix has rank 9 so $L_V$ is unobstructed.
\end{example}

\printbibliography
\Addresses

\end{document}